\documentclass[12pt]{article}
\usepackage{amsmath,amssymb,amsthm,booktabs}
\usepackage{mathrsfs}
\usepackage{geometry}
\usepackage{pifont}
\usepackage[OT1]{fontenc}
\usepackage[utf8]{inputenc}
\usepackage[colorlinks,citecolor=blue,urlcolor=blue]{hyperref}
\usepackage{txfonts}
\usepackage{indentfirst}
\usepackage{multirow}
\usepackage{float,subfig}

\usepackage{bm}
\usepackage{euscript}
\usepackage{graphicx}
\usepackage{multicol}
\usepackage[usenames,dvipsnames,svgnames,table]{xcolor}

\usepackage[numbers]{natbib}
\usepackage[ruled]{algorithm2e}

\numberwithin{equation}{section} \theoremstyle{plain}
\newtheorem{theorem}{Theorem}[section]
\newtheorem{lemma}{Lemma}[section]

\newtheorem{example}{Example}[section]
\newtheorem{definition}{Definition}[section]
\newtheorem{remark}{Remark}[section]

\begin{document}

\newcommand{\gai}[1]{{#1}}


\makeatletter
\def\ps@pprintTitle{%
  \let\@oddhead\@empty
  \let\@evenhead\@empty
  \let\@oddfoot\@empty
  \let\@evenfoot\@oddfoot
}
\makeatother

\newcommand\tabfig[1]{\vskip5mm \centerline{\textsc{Insert #1 around here}}  \vskip5mm}

\vskip2cm

\title{Sublinear expectation structure under countable state space}
\author{Shuzhen Yang\thanks{Shandong University-Zhong Tai Securities Institute for Financial Studies, Shandong University, PR China, (yangsz@sdu.edu.cn). This work was supported by the National Key R\&D program of China (Grant No.2018YFA0703900,\ ZR2019ZD41), National Natural Science Foundation of China (Grant No.11701330), and Taishan Scholar Talent Project Youth Project.}
\quad Wenqing Zhang\thanks{Corresponding author, School of Mathematics, Shandong University, PR China, (zhangwendy@mail.sdu.edu.cn).}
}
\date{}
\maketitle

\begin{abstract}
In this study, we propose the sublinear expectation structure under countable state space. To describe an interesting "nonlinear randomized" trial,  based on a convex compact domain, we introduce a family of probability measures under countable state space. Corresponding the sublinear expectation operator introduced by S. Peng, we consider the related notation under countable state space. Within the countable state framework, the sublinear expectation can be explicitly calculated by a novel repeated summation formula, and some interesting examples are given. Furthermore, we establish Monotone convergence theorem, Fatou's lemma and Dominated convergence theorem of sublinear expectation. Afterwards, we consider the independence under each probability measure, upon which we establish the sublinear law of large numbers and obtain the maximal distribution under sublinear expectation.
\end{abstract}


\section{Introduction}
\label{sec:introduce}

In financial market, the price data is observed at discrete times only \cite{Jean12}. Meanwhile continuous-time process are only approximations to physically realizable phenomena \cite{Bertsimas00}. When the underlying sample path are continuous, the discretely sampled data will always appear as a sequence of discrete jumps \cite{Ait02} indicating that discrete models cannot be derived directly from the discretization of continuous models.
Thus it is crucial to study the financial models under countable time and state. Cox et al. \cite{Cox79} developed the binary tree model and used risk-neutral probabilities to price financial derivatives. Subsequently, there has been a substantial amount of research conducted on financial models using discrete time \cite{Duan95, Heston00, Tao22, Diego23}. For continuous mathematical models of derivative pricing theory, see monograph \cite{Kwok08}.

Mean and volatility uncertainties are two important uncertainty properties in financial market.
To describe the model uncertainty, Peng \cite{Peng1997} first constructed a nonlinear expectation which provides a novel mathematical structure. Furthermore, Peng \cite{Peng2004, Peng2006, Peng2008, Peng2019} originally proposed sublinear expectation space which deduced nonlinear law of large numbers and central limit theorem. Then, sublinear expectation has been widely used in finance \cite{EJ13, EJ14, Peng2022, Peng2023}. Fan \cite{Fan08} considered the Jensen’s inequality for filtration consistent nonlinear expectation without domination condition.

There are many related research focusing on the countable time and state under sublinear expectation. Cohen and Elliott \cite{Cohen11} considered backward stochastic difference equations under discrete time with infinitely state. Belak et al. \cite{Belak18} provided existence, uniqueness, and stability results and established convergence of the associated discrete-time nonlinear aggregations. Grigorova and Li \cite{Grigorova23} studied the stochastic representation problem in discrete time under nonlinear expectation and applied it to the pricing of American options. At the moment, the majority of literatures are rooted in the application of sublinear expectation theory to discrete mathematical models, with limited research commencing with the construction of discrete sublinear expectation structure.

In this paper, we investigate the sublinear expectation structure under countable state space.
We first consider a nonlinear randomized trial, based on which a countable state sample space $\Omega=\{\omega_{i}\}_{i\in \mathbb{Z}^{+}}$ and a family of probability measures $\mathcal{P}_{\Theta} = \{P_{\theta}: {\theta}\in \Theta \}$ are introduced.
In countable state space, a family of probability measures can be characterized by a convex compact domain $\mathcal{D}$, which is denoted by
$$
{\mathcal{D}}=\{(\theta_{i})_{i\in \mathbb{Z}^{+}}:  \underline{f}_{i}(\theta_{1},\cdots,\theta_{i-1})\leq \theta_{i}\leq \overline{f}_{i} (\theta_{1},\cdots,\theta_{i-1}),\ \sum_{i\in Z^+}\theta_i=1\},
$$
where $\{f_{i}\}_{i\in \mathbb{Z}^{+}}$ are continuous functions, the lower bounds $\underline{f}_i$ are convex and the upper bounds $\overline{f}_i$ are concave, satisfying $\overline{f}_i - \underline{f}_i \leq c_i$ with $\{c_i\}_{i\in\mathbb{Z}^+}$ being a positive sequence such that $\sum_{i=1}^\infty c_i < \infty$.
Building on $\mathcal{D}$, the sublinear expectation under countable state space is formulated as
$$
\sup_{\theta \in \mathcal{D}} E_{\theta}[X]:=\sup_{\theta \in \mathcal{D}}\sum_{i\in \mathbb{Z}^{+}} X(\omega_i) P_\theta(\{\omega_i\}).
$$
Obviously, this equation satisfies the properties of the sublinear expectation operator established by Peng \cite{Peng2019}. Indeed, sublinear expectation can be explicitly calculated by a repeated summation formula under countable state space,
$$
\mathbb{E}[X]=\sup_{\theta_1\in I_1} \cdots \sup_{\theta_{i}\in I_{i}}\cdots \left [\sum_{i\in \mathbb{Z}^{+}}X(\omega_i)\theta_i \right],
$$
where $I_{i}=\left [\underline{f}_{i}(\theta_{1}, \cdots, \theta_{i-1}),\ \overline{f}_{i} (\theta_{1},\cdots,\theta_{i-1}) \right],\ i\in \mathbb{Z}^{+}$,
and some related examples are given to verify it.
Furthermore, we present Monotone convergence theorem, Fatou's lemma and Dominated convergence theorem under relatively compact probability set $\mathcal{P}_{\Theta}$.
Afterwards, we proposed a new independence definition of sublinear expectation under each $P_{\theta}$ satisfying
$$
\mathbb{E} [\varphi(X,Y)]= \sup_{\theta\in\mathcal{D}} E_{\theta} \left[E_{\theta} \left [\varphi (x, Y) \right]_{x=X}\right],\quad \forall \varphi\in C_{b.lip}(\mathbb{R}).
$$
Upon this independence and Dominated convergence theorem, we give a new proof for the sublinear law of large numbers under countable state space, wherein the sequence converges to a maximal distribution.

The main contributions of this paper are twofold:

(i). We provide the calculation method of the sublinear expectation under countable state space.
Based on a nonlinear randomized trial, we introduce a countable state space and a family of probability measures. By utilizing a convex compact domain $\mathcal{D}$ to describe a family of probability measures $\mathcal{P}_{\Theta}$, the sublinear expectation can be calculated explicitly.

(ii). We derive some convergence theorems of sublinear expectation, and then deduce the law of large numbers based on the independence under each probability. Within the countable state space, the Monotone convergence theorem, Fatou's lemma and Dominated convergence theorem are established for a relatively compact set of probability measures. Building on these results, we present a novel proof of the law of large numbers under sublinear expectation.

The remainder of this paper is organized as follows. Section \ref{sec:structure} considers the sublinear expectation structure under countable state space and develops a calculation method by a repeated summation formula. Following that, we establish some convergence theorems in a relatively compact probability set $\mathcal{P}_{\Theta}$, and consider the independence under each $P_{\theta}$, from which the law of large numbers is derived in Section \ref{sec:control}. Finally, Section \ref{sec:conclude} concludes this paper and proposes the further study.

\section{Sublinear expectation structure}
\label{sec:structure}

Well-known that the classical randomized trial satisfies the following three properties: (i) We can repeat the trial under the same conditions; (ii) We can obtain all the results of the trial; (iii) We don't know the result of the trial before completing the trial. Since Knight \cite{Knight1921} distinguished the risk (random) and uncertainty in the book "Risk, Uncertainty and Profit", we realize that the classical randomized trial cannot describe the uncertainty in the model. Therefore, in this study, we first introduce a nonlinear randomized trial satisfying the following properties: (i') We cannot repeat the trial under the same conditions; (ii') We can obtain all the results of the trial; (iii') We don't know the result of the trial before completing the trial. The properties (ii') and (iii') of a nonlinear randomized trial are same with that of the classical randomized trial. However, the property (i') shows that there is no deterministic law for the nonlinear randomized trial. Thus, based on the properties (ii') and (iii'), we introduce a countable state sample space $\Omega=\{\omega_{i}\}_{i\in \mathbb{Z}^{+}}$. Based on the property (i'), we consider to use a probability set $\mathcal{P}_{\Theta} = \{P_{\theta}: {\theta}\in \Theta \}$ to describe the uncertainty law of the nonlinear randomized trial.

In the countable state space, we introduce a convex compact domain $\mathcal{D}$ to describe a probability set $\mathcal{P}_{\Theta}$.
A probability set $\mathcal{P}_{\Theta} = \{P_{\theta}: {\theta}\in \Theta \}$ satisfies $\Theta \in \mathcal{A}=\{(\theta_i)_{i\in \mathbb{Z}^{+}}: 0\leq \theta_i\leq 1,\ \sum_{i\in \mathbb{Z}^{+}} \theta_i = 1\}$.
We employ a domain $\mathcal{D}$, a convex compact subset of $\mathcal{A}$, to describe a probability set $\mathcal{P}_{\Theta}$, which is denoted by
\begin{equation}
\label{eq:D}
{\mathcal{D}}=\{(\theta_{i})_{i\in \mathbb{Z}^{+}}:  \underline{f}_{i}(\theta_{1},\cdots,\theta_{i-1})\leq \theta_{i}\leq \overline{f}_{i} (\theta_{1},\cdots,\theta_{i-1}),\ \sum_{i\in \mathbb{Z}^{+}} \theta_i = 1\},
\end{equation}
where $\{f_{i}\}_{i\in \mathbb{Z}^{+}}$ are continuous functions, the lower bounds $\underline{f}_i$ are convex and the upper bounds $\overline{f}_i$ are concave, satisfying $\overline{f}_i - \underline{f}_i \leq c_i$ with $\{c_i\}_{i\in\mathbb{Z}^+}$ being a positive sequence such that $\sum_{i=1}^\infty c_i < \infty$.

\begin{remark}
\label{re:convex_compact}
By constraining the functions $\{f_{i}\}_{i\in \mathbb{Z}^{+}}$, the domain $\mathcal{D}$ in equation (\ref{eq:D}) is convex and compact.
The convexity of $\underline{f}_i$ and concavity of $\overline{f}_i$ yield a convex $\mathcal{D}$.
And compactness of $\mathcal{D}$ is guaranteed under the condition $\overline{f}_i - \underline{f}_i \leq c_i$, where $\{c_i\}_{i\in\mathbb{Z}^+}$ is a positive sequence satisfying $\sum_{i=1}^\infty c_i < \infty$. This follows from the characterization of compact sets in infinite-dimensional product spaces (Theorem 3.28 in \cite{Aliprantis2006}), as the summability condition ensures total boundedness and closedness of $\mathcal{D}$.
\end{remark}

\begin{remark}
When there is no model uncertainty, the lower and upper bounds of the parameter $\theta_i$ in $\mathcal{D}$ coincide, i.e. $\underline{f}_i = \overline{f}_i$ for all $i\in \mathbb{Z}^{+}$. Thus the domain $\mathcal{D}$ degenerates to a singleton set containing only the deterministic probability measure $P$.
\end{remark}

\begin{example}
A finite sample state space represents a canonical special case of a countable state space. For a finite state space $\Omega=\{\omega_1,\dots,\omega_n\}$, the corresponding domain $\mathcal{D}$ defined in (\ref{eq:D}) is an $(n-1)$-dimensional system:
$$
\mathcal{D} = \left\{ (\theta_1,\dots,\theta_{n-1}) \,:\, \underline{f}_1 \leq \theta_1 \leq \overline{f}_1, \ \dots, \ \underline{f}_{n-1}(\theta_1,\dots,\theta_{n-2}) \leq \theta_{n-1} \leq \overline{f}_{n-1}(\theta_1,\dots,\theta_{n-2}) \right\},
$$
where $\underline{f}_i$ and $\overline{f}_i$ are convex and concave, respectively. It is obvious that $\mathcal{D}$ is convex and compact.
To facilitate analysis, the following examples employ a finite state space to explicitly construct the structure of $\mathcal{D}$.

Let $\Omega=\{\omega_{1}, \omega_{2}\}$, the corresponding probability set $\mathcal{P}_{\Theta}$ satisfies $\Theta \in \mathcal{A}=\{(\theta_1,\theta_2): 0\leq \theta_1\leq 1,\ 0\leq \theta_2\leq 1,\ \theta_1+\theta_2=1\}$. An example of the domain is ${\mathcal{D}}=\{\theta_1: 0.2 \le \theta_1 \le 0.5 \}$.

Let $\Omega=\{\omega_{1}, \omega_{2}, \omega_3\}$, the corresponding probability set $\mathcal{P}_{\Theta}$ satisfies $\Theta \in \mathcal{A}=\{(\theta_1, \theta_2, \theta_3): 0\leq \theta_1\leq 1,\ 0\leq \theta_2\leq 1,\ 0\leq \theta_3\leq 1,\ \theta_1+ \theta_2+ \theta_3=1\}$. An example of the domain is ${\mathcal{D}}=\{(\theta_{1},\theta_{2}):  0\leq \theta_1\leq 0.5,\ 0\leq \theta_2\leq 0.5-\theta_1\}$.
\end{example}

With a given countable state space $\Omega=\{\omega_{i}\}_{i\in \mathbb{Z}^{+}}$, let $\mathcal{H}$ be a linear space of real valued functions defined on $\Omega$, and suppose that $\mathcal{H}$ satisfies: $|X|<\infty$ if $X\in\mathcal{H}$. Given a random variable $X\in \mathcal{H}$ satisfying $X(\omega_i)=a_i,\ i\in \mathbb{Z}^{+}$, the expectation of $X$ under each $P_{\theta}$ is expressed as
$$
E_{\theta}[X]=\sum_{i\in \mathbb{Z}^{+}} X(\omega_i) P_\theta(\{\omega_i\}) =\sum_{i\in \mathbb{Z}^{+}} a_i \theta_i.
$$
Note that, there is a family of linear expectations $\{E_{\theta}: \theta \in \mathcal{D}\}$, thus it is natural to consider the upper and lower bounds of linear expectations. Due to the dual relation between the upper and lower bounds of expectations, i.e. $\sup_{\theta \in \mathcal{D}} E_{\theta}[X]=-\inf_{\theta \in \mathcal{D}} E_{\theta}[-X]$, we just concentrate on the upper bounds of expectations, analogous results hold for the lower bounds of expectations. Hence, we introduce the definition of upper expectation, this term can be found in Chapter 6.2 of Peng \cite{Peng2019}.

\begin{definition}[Upper expectation]
\label{de:upperE}
Let $X$ be a random variable defined on the countable state space $(\Omega, \mathcal{H})$. The upper expectation of $X$ is defined as
\begin{equation}
\label{eq:deupper}
\sup_{\theta \in \mathcal{D}} E_{\theta}[X]:=\sup_{\theta \in \mathcal{D}}\sum_{i\in \mathbb{Z}^{+}} X(\omega_i) P_\theta(\{\omega_i\}).
\end{equation}
\end{definition}

\begin{remark}
\label{re:subexp}
Sublinear expectation $\mathbb{E}[\cdot]$ was developed by Peng \cite{Peng2004}. Let $\Omega$ be a given sample space and $\mathcal{H}$ satisfies (1) $c\in \mathcal{H}$ for each constant $c$; (2) $|X|\in\mathcal{H}$ if $X\in\mathcal{H}$. A sublinear expectation $\mathbb{E}$ is a functional $\mathbb{E}[\cdot]: \mathcal{H}\to \mathbb{R}$ satisfing

(i). Monotonicity: \quad $\mathbb{E}[X]\leq \mathbb{E}[Y]$\quad if $X \le Y$;

(ii). Constant preserving: \quad $\mathbb{E}[c]=c$\quad for $c\in \mathbb{R}$;

(iii). Sub-additivity: \quad $\mathbb{E}[X+Y]\leq \mathbb{E}[X]+\mathbb{E}[Y]$;

(iv). Positive homogeneity: \quad $\mathbb{E}[\lambda X]=\lambda \mathbb{E}[X]$\quad for $\lambda\geq 0$.

\noindent In addition, sublinear expectation can be expressed as a supremum of linear expectations
\begin{equation}
\label{eq:represent}
\mathbb{E}[X]=\sup_{\theta\in\Theta}E_{\theta}[X].
\end{equation}

It is obviously that, the upper expectation in Definition \ref{de:upperE} satisfies properties (i)-(iv) as stated. This implies its equivalence to the sublinear expectation framework of Peng \cite{Peng2019} when $\Omega$ is a countable state space equipped with the convex compact domain $\mathcal{D}$ defined in (\ref{eq:D}).
This equivalence justifies the adoption of the term sublinear expectation rather than upper expectation within the following paper.
\end{remark}

\subsection{Calculation of sublinear expectation}

Based on sublinear expectation under countable state space, we develop an explicitly calculation formula through the analytical framework of multiple integrals.
This formula reduces multi-dimensional optimization complexities to tractable one-dimensional formulations, thereby significantly improving computational efficiency.

\begin{theorem}[Repeated summation formula]
\label{theo:calculate}
Let $X$ be a random variable defined on the countable state space $(\Omega, \mathcal{H})$. The sublinear expectation can be calculated by
\begin{equation}
\mathbb{E}[X]=\sup_{\theta_1\in I_1} \cdots \sup_{\theta_{i}\in I_{i}}\cdots \left [\sum_{i\in \mathbb{Z}^{+}}X(\omega_i)\theta_i \right],\quad  i\in \mathbb{Z}^{+},
\end{equation}
where $I_{i}$ denotes the projection constraint of $\mathcal{D}$ with $I_{i}=\left [\underline{f}_{i}(\theta_{1}, \cdots, \theta_{i-1}),\ \overline{f}_{i} (\theta_{1},\cdots,\theta_{i-1}) \right]$.
\end{theorem}

\begin{proof}
From Definition \ref{de:upperE}, we have
$$
\mathbb{E}[X]=\sup_{\theta \in \mathcal{D}} E_{\theta}[X] =\sup_{\theta \in \mathcal{D}}\sum_{i\in \mathbb{Z}^{+}} X(\omega_i) P_\theta(\{\omega_i\})=\sup_{\theta \in \mathcal{D}}\sum_{i\in \mathbb{Z}^{+}} X(\omega_i) \theta_i.
$$
Let
$$
G(\theta)=\sum_{i\in \mathbb{Z}^{+}} X(\omega_i) \theta_i.
$$
On the one hand, let $y =\sup_{\theta \in \mathcal{D}}G(\theta)$. Thus for each $(\theta_{i})_{i\in Z^+}\in \mathcal{D}$, the iterated supremum satisfies
\begin{equation}
\label{eq:cal1}
\sup_{\theta_1\in I_1} \cdots \sup_{\theta_{i}\in I_{i}} \cdots G(\theta) \le y=\sup_{\theta \in \mathcal{D}}G(\theta), \quad i\in \mathbb{Z}^{+}.
\end{equation}
On the contrary, let $z=\sup_{\theta_1\in I_1} \cdots \sup_{\theta_{i}\in I_{i}}\cdots G(\theta),\ i\in \mathbb{Z}^{+}$. Since $\mathcal{D}$ is convex and compact, the supremum $z$ is attained at some $\theta_i^*\in I_i, i\in \mathbb{Z}^{+}$, i.e., $z =  G(\theta_1^*, \cdots, \theta_{i}^*,\cdots)$. Then for any $\theta_i$, $G(\theta) \le z$, it is obvious that
\begin{equation}
\label{eq:cal2}
\sup_{\theta \in \mathcal{D}}G(\theta) \le z = \sup_{\theta_1\in I_1} \cdots \sup_{\theta_{i}\in I_{i}}\cdots G(\theta),\quad i\in \mathbb{Z}^{+}.
\end{equation}
Combining inequalities (\ref{eq:cal1}) and (\ref{eq:cal2}), we conclude
$$
\sup_{\theta \in \mathcal{D}} E_{\theta}[X]=\sup_{\theta_1\in I_1} \cdots \sup_{\theta_{i}\in I_{i}}\cdots \left [\sum_{i\in \mathbb{Z}^{+}}X(\omega_i)\theta_i \right],\quad i\in \mathbb{Z}^{+}.
$$
This completes the proof.
\end{proof}

In the following, we present two examples to verify Theorem \ref{theo:calculate}.
\begin{example}
\label{exam:X_size}
Let $\Omega =\{\omega_1,\omega_2\}$ and domain ${\mathcal{D}}= \{\theta_1: 0.2 \le \theta_1 \le 0.5 \}$. By Theorem \ref{theo:calculate}, if $X(\omega_1)\ge X(\omega_2)$, the sublinear expectation can be calculated as follows
\begin{align*}
\mathbb{E}[X]=\sup_{\theta_1 \in [0.2,0.5]} [(a_1-a_2) \theta_1 + a_2]=0.5 a_1 +0.5 a_2.
\end{align*}
Conversely, if $X(\omega_1)<X(\omega_2)$, we derive $\mathbb{E}[X]=0.2 a_1+0.8 a_2$. For the sake of convenience, we adopt a fixed ordering $X(\omega_1)\ge X(\omega_2)\ge \cdots$ in subsequent examples. Analogous results hold for alternative ordering configurations of $X(\omega_i)$.

Let $\Omega =\{\omega_1,\omega_2,\omega_3\}$ and domain ${\mathcal{D}}=\{(\theta_{1},\theta_{2}):  0\leq \theta_1\leq 0.5,\ 0\leq \theta_2\leq 0.5-\theta_1\}$. Without lose of generality, we assume that $X(\omega_1) > X(\omega_2) > X(\omega_3)$. Based on Theorem \ref{theo:calculate}, the sublinear expectation can be calculated as follows
\begin{align*}
\mathbb{E}[X]
&=\sup_{\theta_1 \in [0,0.5]} \sup_{\theta_2 \in [0,0.5-{\theta_1}]} [(a_1-a_3)\theta_1+ (a_2-a_3)\theta_2]+a_3\\
&=\sup_{\theta_1 \in [0,0.5]} [(a_1-a_3)\theta_1+ (a_2-a_3) (0.5-\theta_1)]+a_3 =0.5 (a_1 + a_3).
\end{align*}
\end{example}

From the above examples, it is evident that when the boundary of domain $\mathcal{D}$ is linear, that is $\underline{f}_i$ and $\overline{f}_i$ are linear functions, the supremum of $E_{\theta}[X]$ is achieved on the boundary.
Conversely, for nonlinear boundaries $\underline{f}_i$ or $\overline{f}_i$, the irregular domain $\mathcal{D}$ can be transformed into a canonical rectangular domain $\mathcal{M}$ via variable substitution.
This reduction enables the supremum to be localized on the boundary of $\mathcal{M}$, substantially mitigating computational complexity. We then derive the coordinate transform formula to operationalize this framework.

\begin{lemma}[Transform formula]
\label{lemma:transform}
Let $h(\cdot)$ and $g(\cdot)$ be continuous functions, and let $\mathcal{D}=\left \{(\theta_{i})_{i\in \mathbb{Z}^{+}}\right \}$ and $\mathcal{M}=\left \{(\delta_{i})_{i\in \mathbb{Z}^{+}}\right \} $ be two convex compact domains.
Consider a componentwise bijection $T:\mathcal{M}\to\mathcal{D}$ defined by $\theta_{i}=\theta_{i}(\delta_1, \cdots, \delta_{i})$ for each $i\in \mathbb{Z}^{+}$.
Assume the Jacobian determinant of the $i$-dimensional projection satisfies:
$$
J(\delta_1,\cdots, \delta_{i})=\frac{\partial (\theta_1,\cdots,\theta_{i})}{\partial (\delta_1,\cdots, \delta_{i})} \ne 0,\quad \forall i\in\mathbb{Z}^+.
$$
Then
$$
\sup_{\theta_1\in I_1} \cdots \sup_{\theta_{i}\in I_{i}} \cdots h(\theta_1,\cdots,\theta_{i}, \cdots)
=\sup_{\delta_1\in J_1 } \cdots \sup_{\delta_{i} \in J_{i}} \cdots g(\delta_1,\cdots,\delta_{i}, \cdots),\quad i\in\mathbb{Z}^{+},
$$
where $I_i$ and $J_i$ denote the projection constraints of $\mathcal{D}$ and $\mathcal{M}$, respectively.
\end{lemma}

\begin{proof}
Let
$$
y=\sup_{\theta_1\in I_1} \cdots \sup_{\theta_{i}\in I_{i}} \cdots h(\theta_1,\cdots, \theta_{i},\cdots),\quad i\in \mathbb{Z}^{+}.
$$
Owing $\mathcal{D}$ is a convex compact domain and $h$ is continuous, the supremum value of $y$ is attained at some $\theta_i^* \in I_i$ for $i\in\mathbb{Z}^+$, that is $y = h(\theta_1^*, \cdots, \theta_{i}^*, \cdots), i\in \mathbb{Z}^{+}$.
Since $T$ is a bijection, then
$$
h(\theta_1,\cdots,\theta_{i},\cdots)=h(\theta_1(\delta_1, \cdots, \delta_{i}), \cdots, \theta_{i}(\delta_1, \cdots, \delta_{i}),\cdots)=g(\delta_1,\cdots,\delta_{i},\cdots),\quad i\in Z^+,
$$
where $g(\delta_{i})=h(T)$.
Since $h(\theta_1,\cdots, \theta_{i},\cdots) \le y$, we have $g(\delta_1,\cdots,\delta_{i},\cdots) \le y$. Thus
\begin{equation}
\label{eq:trans1}
\sup_{\delta_1\in J_1 } \cdots \sup_{\delta_{i} \in J_{i}}\cdots g(\delta_1,\cdots,\delta_{i},\cdots) \le y = \sup_{\theta_1\in I_1} \cdots \sup_{\theta_{i}\in I_{i}}\cdots h(\theta_1,\cdots, \theta_{i},\cdots),\quad i\in Z^+.
\end{equation}
On the contrary, let
$$
z=\sup_{\delta_1\in J_1 } \cdots \sup_{\delta_{i} \in J_{i}}\cdots g(\delta_1,\cdots,\delta_{i},\cdots),\quad i\in Z^+.
$$
Using the similar manner in the proof of inequality (\ref{eq:trans1}), we have $y\leq z$, which completes the proof.
\end{proof}

\begin{example}
Let $\Omega =\{\omega_1,\omega_2,\omega_3\}$ and ${\mathcal{D}}=\{(\theta_{1},\theta_{2}):  0\leq \theta_1\leq 0.5,\ 0\leq \theta_2\leq \sqrt{\theta_1}\}$. Let $\delta_1=\theta_1, \delta_2= \frac{\theta_{2}^{2}} {\theta_1}$ (if $\theta_1=0$, let $\delta_2=0$), we can transform a irregular domain $\mathcal{D}$ into a rectangular domain ${\mathcal{M}}=\{(\delta_1,\delta_2):  0\leq \delta_1\leq 0.5,\ 0\leq \delta_2\leq 1\}$. Without loss of generality, we assume that $X(\omega_1)>X(\omega_2)>X(\omega_3)$, based on Theorem \ref{theo:calculate} and Lemma \ref{lemma:transform}, the sublinear expectation can be calculated as follows
\begin{align*}
\mathbb{E}[X]
&=\sup_{\theta_1 \in [0,0.5]} \sup_{\theta_2 \in [0,\sqrt{\theta_1}]}[(a_1-a_3)\theta_1+ (a_2-a_3)\theta_2]+ a_3\\
&=\sup_{\delta_1 \in [0,0.5]} \sup_{\delta_2 \in [0,1]} [(a_1-a_3)\delta_1+ (a_2-a_3)\sqrt{\delta_1 \delta_2}]+ a_3\\
&=\sup_{\delta_1 \in [0,0.5]} [(a_1-a_3)\delta_1+ (a_2-a_3)\sqrt{\delta_1}] +a_3=\frac{1}{2}a_1+ \frac{\sqrt{2}}{2}a_2+ \frac{ {1-\sqrt{2}}}{2}a_3.
\end{align*}

In the following, we consider another example where the domain $\mathcal{D}$ is a circle, i.e. \\ $ \mathcal{D}= \left \{\left(\theta_1,\theta_2 \right) : \left(\theta_1-0.25 \right)^2+ \left(\theta_2-0.25\right)^2 \le \left(0.25\right)^2, \theta_1,\theta_2\geq 0 \right\}$.
According to the polar coordinate transformation formula,
$$
\left\{\begin{matrix}
\theta_1=\gamma \cdot \cos(\delta)+\frac{1}{4}, \\
\theta_2=\gamma \cdot \sin(\delta)+\frac{1}{4},
\end{matrix}\right.
$$
we can transform a circle domain $\mathcal{D}$ into a rectangular domain  ${\mathcal{M}}=\left \{\left(\gamma,\delta \right):  0\leq \gamma \leq 0.25,\ 0\leq \delta \leq 2\pi \right\}$. We assume that $X(\omega_1)=2,\ X(\omega_2)=2,\ X(\omega_3)=1$, based on Theorem \ref{theo:calculate} and Lemma \ref{lemma:transform}, the sublinear expectation can be calculated as follows
\begin{align*}
\mathbb{E}[X]
&=\sup_{\theta_1 \in [0,0.5]} \sup_{\theta_2 \in [0,\sqrt{\theta_1 \left(0.5-\theta_1 \right)}+ 0.25]} [(a_1-a_3)\theta_1+ (a_2-a_3)\theta_2]+ a_3\\
&=\sup_{\gamma \in [0,0.25]} \sup_{\delta \in [0,2\pi]} [(a_1-a_3)\gamma \cos(\delta)+ (a_2-a_3)\gamma \sin(\delta)]+ \frac{1}{2}+ a_3\\
&=\sup_{\gamma \in [0,0.25]} \sqrt{2} \gamma+\frac{3}{2}=\frac{6+\sqrt{2}}{4}.
\end{align*}
\end{example}

\section{Convergence theorems and Law of large numbers}
\label{sec:control}

\subsection{Convergence theorems}

Within the framework of sublinear expectations under countable state space, we establish the Monotone convergence theorem, Fatou's lemma and Dominated convergence theorem below.
Prior work by Denis et al. \cite{Peng2011} established a Monotone convergence theorem under sublinear expectation, asserting that if $\mathcal{P}_{\Theta}$ is weakly compact and $\{X_m\}_{m\ge 1}\subset \mathbb{L}_{c}^{1}$ satisfies $X_m\downarrow X, q.s.$, then $\mathbb{E}[X_m]\downarrow \mathbb{E}[X]$.
And Lemma 7-8 in \cite{Peng2011} further demonstrated that weak compactness strictly implies relative compactness.
Our results relax the topological requirement on $\mathcal{P}_{\Theta}$ from weak compactness to relative compactness, thereby broadening the scope of applicability.

\begin{remark}
\label{re:rela_compact}
The convexity and compactness of domain $\mathcal{D}$ in equation (\ref{eq:D}) ensure the relative compactness of probability sets $\mathcal{P}_{\Theta}$.
By assumption, since $\overline{f}_i - \underline{f}_i \le c_i$ and $\underline{f}_i \le \theta_i \le \overline{f}_i$, we have $\theta_i \le \underline{f}_i + c_i$ for all $\theta\in\mathcal{D}$.
The convergence of $\sum_{i=1}^{\infty} \underline{f}_i$ follows from $\sum_{i=1}^{\infty} \underline{f}_i \le \sum_{i=1}^{\infty} \theta_i =1$. For any $\epsilon>0$, choose $N_1, N_2 \in \mathbb{Z}^{+}$ such that $\sum_{i=N_1}^{\infty} c_i< \frac{\epsilon}{2},\ \sum_{i=N_2}^{\infty} \underline{f}_i< \frac{\epsilon}{2}$.
Let $N=\max \{N_1, N_2\}$, define the finite set $K=\left \{\omega_1,\cdots, \omega_{N-1} \right \}$, which is obviously compact. For its complement $K^c=\left \{\omega_{N}, \omega_{N+1}, \cdots \right \}$, the tail probability satisfies
$$
P_{\theta}(K^c)=\sum_{i=N}^{\infty} \theta_i \le \sum_{i=N}^{\infty} (\underline{f}_i+c_i) \le \sum_{i=N}^{\infty} \underline{f}_i + \sum_{i=N}^{\infty} c_i <\epsilon,
$$
for all $P_{\theta}\in\mathcal{P}_{\Theta}$. Thus, $\sup_{P_{\theta}\in\mathcal{P}} P_{\theta}(K^c)< \epsilon$.
Hence, $\mathcal{P}_{\Theta}$ is relatively compact by Theorem 6 in \cite{Peng2011}.
\end{remark}

As a prerequisite for proving the Monotone convergence theorem, we establish the equivalence between the relative compactness of probability sets $\mathcal{P}_{\Theta}$ and the regularity of the associated sublinear expectation $\mathbb{E}[\cdot]$ under countable state space. Here, regularity is defined by the condition that for each bounded sequence $\{X_m\}_{m\ge 1}$ satisfying $X_m\downarrow 0, q.s.$, it follows that $\mathbb{E}[X_m]\downarrow 0$.

\begin{lemma}
\label{lemma:regular}
Under the countable state space $(\Omega, \mathcal{H})$, $\mathbb{E}[\cdot]$ is regular if and only if $\mathcal{P}_{\Theta}$ is relatively compact.
\end{lemma}
\begin{proof}
As for "if" statement, by relative compactness of $\mathcal{P}_{\Theta}$, there exists a compact set $K$ in the countable state space $\Omega$ such that $\sup_{P_{\theta}\in \mathcal{P}_{\Theta}} P_{\theta}(K^c)<\epsilon,\ \forall \epsilon >0$.
Let $\{X_m\}_{m\ge 1}$ be a bounded sequence with $X_m \downarrow 0$ and $|X_m|\le C$ for constant $0\le C<\infty$.
It is obviously that $X_m\cdot I_K\downarrow 0$ uniformly, then there exists $N\in\mathbb{Z}^{+}$ such that $X_m\cdot I_K< \epsilon$ for all $m\ge N$.
We have
$$
\mathbb{E}[X_m]\le \mathbb{E}[X_m\cdot I_K]+\mathbb{E}[X_m \cdot I_{K^c}]\le \epsilon + C\cdot \sup_{P_{\theta}\in\mathcal{P}_{\Theta}} P_{\theta}(K^c) \le (1+C)\epsilon.
$$
As $\epsilon>0$ is arbitrary, we have $\mathbb{E}[X_m]\downarrow 0$.

On the other hand, assume $\mathbb{E}$ is regular. under countable state space, there exists an increasing compact set sequence $\{K_n\}_{n\ge 1}$ such that $\bigcup_{n=1}^\infty K_n = \Omega$. Define $X_n = I_{K_n^c}$, which satisfies $X_n \downarrow 0$ and $X_n$ is bounded.
By regularity,
$$
\sup_{P_{\theta}\in\mathcal{P}_{\Theta}} P_{\theta}(K_n^c) = \mathbb{E}[X_n] \downarrow 0.
$$
Hence, for any $\epsilon > 0$, there exists $N\in\mathbb{Z}^{+}$ such that $\sup_{P_{\theta}\in\mathcal{P}_{\Theta}} P(K_N^c) < \epsilon$, verifying that $\mathcal{P}_{\Theta}$ is relatively compact.
\end{proof}

\begin{remark}
Theorem 12 in \cite{Peng2011} constructed an analogous equivalence under the restriction to continuous bounded random variables.
Lemma \ref{lemma:regular} extends this result by removing  the continuity assumption, thereby generalizing the scope to all bounded measurable random variables.
\end{remark}

\begin{theorem}[Monotone convergence theorem for bounded random variables]
\label{theo:monotone}
Let $\{X_m\}_{m\geq 1}$ and $X$ be a bounded random sequence and variable, respectively, defined on the countable state space $(\Omega, \mathcal{H})$.

(1) Let $X_m\uparrow X, q.s.$ Then $\mathbb{E}[X_m]\uparrow \mathbb{E}[X]$.

(2) Let $X_m\downarrow X, q.s.$ Then $\mathbb{E}[X_m]\downarrow \mathbb{E}[X]$.
\end{theorem}

\begin{proof}
(1) Since $X_m\uparrow X, q.s.$, the monotonicity of $\mathbb{E}$ implies $\mathbb{E}[X_m]\le \mathbb{E}[X]$ for all $m$, hence $\lim_{m\to\infty} \mathbb{E} [X_m]\le \mathbb{E}[X]$.
By Remark \ref{re:rela_compact} and Lemma \ref{lemma:regular}, the regular of $\mathbb{E}[\cdot]$ ensures that $X-X_m\downarrow 0, q.s.$ implies $\mathbb{E}[X-X_m]\downarrow 0$. Therefore,
$$
0=\lim_{m\to\infty}\mathbb{E}[X-X_m]\ge \mathbb{E}[X]-\lim_{m\to\infty}\mathbb{E}[X_m],
$$
which yields $\lim_{m\to\infty} \mathbb{E}[X_m]\ge \mathbb{E}[X]$. Combining inequalities, we conclude $\mathbb{E}[X_m]\uparrow \mathbb{E}[X]$.

(2) As the same way in the proof of (1), we can deduce from $X_m\downarrow X$ that $\mathbb{E}[X_m]\ge \mathbb{E} [X]$, giving $\lim_{m\to\infty} \mathbb{E}[X_m]\ge \mathbb{E}[X]$.
By regularity, $X_m - X \downarrow 0, q.s.$ leads to $\mathbb{E}\left[X_m-X\right]\downarrow 0$, hence
$$
0=\lim_{m\to\infty}\mathbb{E}[X_m-X]\ge \lim_{m\to\infty}\mathbb{E}[X_m]-\mathbb{E}[X],
$$
which implies $\lim_{m\to\infty} \mathbb{E}[X_m]\le \mathbb{E}[X]$. Thus, we can deduce that $\mathbb{E}[X_m]\downarrow \mathbb{E}[X]$, which completes the proof.
\end{proof}

\begin{remark}
The Monotone convergence theorem for bounded random variables is less affected by the state of sample space. Specifically, Theorem \ref{theo:monotone} extends naturally to uncountable state space $(\Omega, \mathcal{H})$.
\end{remark}

The bounded condition of random variables in Theorem \ref{theo:monotone} can be generalized to $\mathbb{L}_b^1$ condition, defined by
$$
\mathbb{L}_b^1 = \left \{X\in\mathbb{L}^{1}: \lim_{n\to\infty} \mathbb{E}[|X| I_{\{|X|>n\}}]=0\right \},
$$
where $\mathbb{L}^{1}$ denotes the Banach space endowed with the norm
$$
\left \|X\right \|_{1}:=\mathbb{E}[|X|]=\sup_{\theta\in\mathcal{D}} E_{\theta}[|X|]< \infty.
$$
We thereby relax the boundedness constraint of random variables while preserving convergence, which formalize this extension as following theorem.

\begin{theorem}[Monotone convergence theorem for $\mathbb{L}_b^1$ random variables]
\label{theo:monotone2}
Let $\{X_m\}_{m\geq 1}$ and $X$ be a $\mathbb{L}_b^1$-random sequence and variable, respectively, defined on the countable state space $(\Omega, \mathcal{H})$.

(1) Let $X_m\uparrow X, q.s.$ Then $\mathbb{E}[X_m]\uparrow \mathbb{E}[X]$.
	
(2) Let $X_m\downarrow X, q.s.$ Then $\mathbb{E}[X_m]\downarrow \mathbb{E}[X]$.
\end{theorem}

\begin{proof}
We start with the proof of (2). Let $\bar{\mathcal{P}}_{\Theta}$ be the closure of relatively compact set $\mathcal{P}_{\Theta}$, and $\bar{\mathbb{E}}$ is the corresponding sublinear expectation.
Since $\bar{\mathcal{P}}_{\Theta}$ is compact, it is also weakly compact.
On the one hand, for any $X\in \mathbb{L}_b^1$, we have
\begin{equation}
\label{eq:theo2}
\bar{\mathbb{E}}[X]=\sup_{\bar{P}_{\theta}\in\bar{\mathcal{P}}_{\Theta}} E_{\bar{P}_{\theta}}[X]\ge \sup_{P_{\theta}\in\mathcal{P}_{\Theta}} E_{P_{\theta}}[X]=\mathbb{E}[X].
\end{equation}
On the other side, there exists a sequence $\{\bar{P}_{n}\}\subset \bar{\mathcal{P}}_{\Theta}$ converging weakly to a $\bar{P}\in \bar{\mathcal{P}}_{\Theta}$, with $\bar{\mathbb{E}}[X] = E_{\bar{P}}[X]$.
And there exists a sequence $\{P_{n}\}\subset\mathcal{P}_{\Theta}$ such that $d(\bar{P}_{n}, P_{n})\le \frac{1}{n}$ for all $n$. Applying the triangle inequality, we obtain
$$
d(\bar{P}, P_{n})\le d(\bar{P}, \bar{P}_{n})+ d(\bar{P}_{n}, P_{n})\le d(\bar{P}, \bar{P}_{n})+\frac{1}{n}.
$$
Since $\lim_{n\to\infty} d(\bar{P}, \bar{P}_{n})= 0$ by weak convergence, we have $\lim_{n\to\infty} d(\bar{P}, P_{n})= 0$. Then it follows that,
\begin{equation}
\label{eq:theo3}
\bar{\mathbb{E}} [X]=E_{\bar{P}}[X]=\lim_{n\to\infty} E_{P_{n}}[X]\le \mathbb{E}[X].
\end{equation}
Combing equations (\ref{eq:theo2}) and (\ref{eq:theo3}), we conclude $\mathbb{E}[X]=\bar{\mathbb{E}}[X]$ for all $X\in\mathbb{L}_{b}^{1}$ in the countable state space $(\Omega, \mathcal{H})$.
Then, we can infer that (2) follows from Theorem 31 in \cite{Peng2011}.

As for (1), $X_m\uparrow X$ implies $\mathbb{E}[X_m]\le \mathbb{E}[X]$, so $\lim_{m\to\infty} \mathbb{E}[X_m]\le \mathbb{E}[X]$.
On the contrary, note that $X-X_m\in \mathbb{L}_b^1$ because
$$
\mathbb{E}[|X-X_m|]\le \mathbb{E}[|X|]+\mathbb{E}[|X_m|]\le \infty,
$$
and
$$
0\le \lim_{n\to\infty}\mathbb{E}[|X-X_m| I_{\{|X-X_m|>n\}}] \le \lim_{n\to\infty}\mathbb{E}[|X| I_{\{|X|>\frac{n}{2}\}}] + \lim_{n\to\infty}\mathbb{E}[|X_m| I_{\{|X_m|>\frac{n}{2}\}}]\le \lim_{n\to\infty}\mathbb{E}[|X| I_{\{|X|>\frac{n}{2}\}}]=0,
$$
where the last equality follows from $X\in\mathbb{L}_b^1$.
Since $X-X_m\downarrow 0$, (2) gives  $\mathbb{E}[X-X_m]\downarrow 0$. Therefore,
$$
0=\lim_{m\to\infty}\mathbb{E}[X-X_m]\ge \mathbb{E}[X]-\lim_{m\to\infty}\mathbb{E}[X_m],
$$
implying $\lim_{m\to\infty} \mathbb{E}[X_m]\ge \mathbb{E}[X]$.
Hence, $\mathbb{E}[X_m]\uparrow \mathbb{E}[X]$, which completes the proof.
\end{proof}

\begin{remark}
The requirement $X \in \mathbb{L}_b^1$ in Theorem \ref{theo:monotone2} (2) can be omitted when $X_m \downarrow X, q.s.$, as it is inherited from the sequence $\{X_m\}_{m\ge 1} \subset \mathbb{L}_b^1$. Specifically, since $|X| \leq |X_m|$ for all $m$ and $\lim_{n \to \infty} \mathbb{E}[|X| I_{\{|X| > n\}}] \leq \lim_{n \to \infty} \mathbb{E}[|X_m| I_{\{|X_m| > n\}}] = 0$, the limit $X$ satisfies the $\mathbb{L}_b^1$-condition.
Thus, Theorem \ref{theo:monotone2} (2) generalizes Theorem 31 in \cite{Peng2011} by replacing the weak compactness of probability set with the weaker assumption of relative compactness.
\end{remark}

\begin{lemma}[Fatou's lemma]
\label{lemma:fatou}
Let $\{X_m\}_{m\geq 1}\subset \mathbb{L}_b^1$ be a random sequence defined on the countable state space $(\Omega, \mathcal{H})$.

(1) If there exists a random variable $Y\in \mathbb{L}_b^1$ such that $X_m \ge Y$ for all $m\ge 1$, and $\underline{\lim}_{m\to\infty}X_m \in \mathbb{L}_b^1$. Then
$$
\mathbb{E} [\underline{\lim}_{m\to\infty} X_m] \le \underline{\lim}_{m\to\infty} \mathbb{E}[X_m].
$$

(2) If there exists a random variable $Y\in \mathbb{L}_b^1$ such that $X_m \le Y$ for all $m\ge 1$, and $\overline{\lim}_{m\to\infty}X_m \in \mathbb{L}_b^1$. Then
$$
\mathbb{E}[\overline{\lim}_{m\to\infty} X_m] \ge \overline{\lim}_{m\to\infty} \mathbb{E}[X_m].
$$
\end{lemma}

\begin{proof}
Regarding (1), let $g_m=\inf_{k \ge m} X_k$, we have $g_m \uparrow \underline{\lim}_{m\to\infty}X_m$ owing $X_m \ge Y, \forall m\ge 1$.
Since $\{X_m\}_{m\ge 1}, Y\in \mathbb{L}_b^1$ and $Y\le g_m \le X_m$, then
\begin{equation}
\label{eq:lemma1}
\mathbb{E}[|g_m|]\le \mathbb{E}[|Y|+|X_m|]<\infty,
\end{equation}
and
\begin{equation}
\label{eq:lemma2}
0\le \lim_{n\to\infty} \mathbb{E}[|g_m| I_{\{|g_m|>n \}}]
\le \lim_{n\to\infty} \mathbb{E}[|Y|I_{\{|Y|>\frac{n}{2}\}}] + \lim_{n\to\infty} \mathbb{E}[|X_m|I_{\{|X_m|>\frac{n}{2}\}}]
\le \lim_{n\to\infty} \mathbb{E}[|X_m|I_{\{|X_m|>n\}}] =0,
\end{equation}
indicating $g_m\in\mathbb{L}_b^1$.
Then applying (1) in Theorem \ref{theo:monotone2}, we have
$$
\mathbb{E}[\underline{\lim}_{m\to\infty}X_m] = \lim_{m \to \infty}\mathbb{E}[g_m] \le \underline{\lim}_{m\to\infty} \mathbb{E}[X_m].
$$

As for the proof of (2), let $h_m=\sup_{k \ge m} X_k$, then $h_m \downarrow \overline{\lim}_{m\to\infty}X_m$. The $h_m\in \mathbb{L}_b^1$ can be derived by an argument analogous to the method applied in equations (\ref{eq:lemma1}) and (\ref{eq:lemma2}).
By (2) in Theorem \ref{theo:monotone2}, we have
$$
\mathbb{E} [\overline{\lim}_{m\to\infty} X_m] = \lim_{m \to \infty}\mathbb{E}[h_m] \ge \overline{\lim}_{m\to\infty} \mathbb{E}[X_m].
$$
This completes the proof.
\end{proof}

\begin{theorem}[Dominated convergence theorem]
\label{theo:control}
Let $\{X_m\}_{m\geq 1}$ and $X$ be a random sequence and variable, respectively, defined on the countable state space $(\Omega, \mathcal{H})$. Since $X_m\to X, q.s.$, and there is a non-negative random variable $Y\in\mathbb{L}_b^1$ such that $|X_m|\le Y$ for all $m\ge 1$. Then
$$
\lim_{m \to \infty} \mathbb{E}[X_m]=\mathbb{E}[X].
$$
\end{theorem}

\begin{proof}
Given the uniform bound $|X_m| \leq Y$ for all $m \ge 1$, the limit $X$ satisfies $|X| \leq Y$. Since $Y \in \mathbb{L}_b^1$, then $\{X_m\}_{m\geq 1}\subset \mathbb{L}_b^1$ and $X \in \mathbb{L}_b^1$.
It is apparent that
$$
X=\lim_{m\to\infty}X_m=\underline{\lim}_{m\to\infty}X_m= \overline{\lim}_{m\to\infty}X_m,
$$
Applying Lemma \ref{lemma:fatou}, we can obtain that
$$
\mathbb{E}[X]=\mathbb{E}[\underline{\lim}_{m\to\infty} X_m]\le \underline{\lim}_{m\to\infty} \mathbb{E}[X_m] \le \overline{\lim}_{m\to\infty} \mathbb{E}[X_m] \le \mathbb{E} [\overline{\lim}_{m\to\infty} X_m] =\mathbb{E}[X].
$$
This implies $\mathbb{E}[X]=\lim_{m\to\infty} \mathbb{E}[X_m]$, which completes the proof.
\end{proof}

\subsection{Law of large numbers}
In countable state space, the law of large numbers under sublinear expectation can be established via the Dominated convergence theorem. Before presenting this results, we provide key definitions of identically distributed and independence.
The definition of identically distributed under countable state space aligns with Definition 1.3.1 in Peng's classical framework \cite{Peng2019}. Two random variables $X_1$ and $X_2$ defined on the countable state space $(\Omega_1,\mathcal{H}_1)$ and $(\Omega_2,\mathcal{H}_2)$, respectively, are called identically distributed, denoted by $X_1 \overset{d}{=} X_2$, if
\begin{equation}
\label{eq:subdistri}
\mathbb{E}_1 [\varphi (X_1)]=\mathbb{E}_2 [\varphi (X_2)],\quad \forall \varphi \in C_{b.Lip}(\mathbb{R}).
\end{equation}
where $C_{b.Lip}(\mathbb{R})$ denotes the space of bounded lipschitz functions.
Regarding the definition of independence, in Definition 1.3.11 of Peng's framework \cite{Peng2019}, a random variable $Y$ is said to be independent of $X$ if
\begin{equation}
\label{eq:cla_indepen}
\mathbb{E}[\varphi (X,Y)]=\mathbb{E}[\mathbb{E}[\varphi(x,Y)]_{x=X}],\quad \forall \varphi \in C_{b.Lip}(\mathbb{R}^2).
\end{equation}
This definition is asymmetric, that is, independence of $Y$ from $X$ does not imply the converse. While under countable state space, the independence structure of sublinear expectation admits mutual symmetry.
To formalize this, we introduce a new definition of independence for each $P_{\theta}$ on countable state space.

\begin{definition}[Independence]
\label{de:subindepen}
Let $X$ and $Y$ be two random variables defined on the countable state space $(\Omega,\mathcal{H})$.
We call $Y$ is independent of $X$, if
\begin{equation}
\label{eq:subindepen}
\mathbb{E} [\varphi(X,Y)] = \sup_{\theta\in\mathcal{D}} E_{\theta} \left[E_{\theta} \left [\varphi (x, Y) \right]_{x=X}\right],\quad \forall \varphi \in C_{b.lip}(\mathbb{R}^2).
\end{equation}
\end{definition}

\begin{remark}
\label{re:P_theta}
The independence under each probability $P_{\theta}$ implies equation (\ref{eq:subindepen}) in Definition \ref{de:subindepen}.
Indeed, observe that:
\begin{align}
\label{eq:remark3}
\mathbb{E} [\varphi(X,Y)]&=\sup_{\theta\in\mathcal{D}} E_{\theta}[\varphi(X,Y)]=\sup_{\theta\in\mathcal{D}} \sum_{i, j\in \mathbb{Z}^{+}} \varphi(X(\omega_{i}), Y(\omega_{j})) P_{\theta}(X=X(\omega_{i}),Y=Y(\omega_{j})),
\end{align}
and
\begin{align}
\label{eq:remark4}
\sup_{\theta\in\mathcal{D}} E_{\theta} \left[E_{\theta} \left [\varphi (x, Y) \right]_{x=X}\right]&= \sup_{\theta\in\mathcal{D}} E_{\theta} \left[\sum_{j\in \mathbb{Z}^{+}} \varphi (x, Y(\omega_{j})) P_{\theta}(Y=Y(\omega_{j})) \ \Big\vert_{x=X} \right]\notag\\
&=\sup_{\theta\in\mathcal{D}}  \left[\sum_{i\in \mathbb{Z}^{+}} \sum_{j\in \mathbb{Z}^{+}} \varphi (X(\omega_{i}), Y(\omega_{j})) P_{\theta}(X=X(\omega_{i})) P_{\theta}(Y=Y(\omega_{j})) \right].
\end{align}
By the independence under each $P_{\theta}$, we have
\begin{equation}
\label{eq:remark5}
P_{\theta}(X=X(\omega_{i}),Y=Y(\omega_{j}))=P_{\theta}(X=X(\omega_{i})) P_{\theta}(Y=Y(\omega_{j})), \quad \forall \theta\in\mathcal{D},\ \forall i,j \in\mathbb{Z}^{+}.
\end{equation}
Substituting (\ref{eq:remark5}) into (\ref{eq:remark3}), the expression in (\ref{eq:remark3}) and (\ref{eq:remark4}) coincide.
\end{remark}

\begin{remark}
\label{re:independent}
The independence in Definition \ref{de:subindepen} is symmetric, i.e., if $Y$ is independent of $X$, then $X$ is independent of $Y$, and vice versa.
And the relationship between Definition \ref{de:subindepen} and Definition 1.3.11 of \cite{Peng2019} can be formalized through the inequality:
\begin{equation}
\label{eq:rela_indepen}
\sup_{\theta\in\mathcal{D}} E_{\theta} \left[E_{\theta} \left [\varphi (x, Y) \right]_{x=X}\right]
\le \sup_{\theta\in\mathcal{D}} E_{\theta} \left[\sup_{\theta\in\mathcal{D}} E_{\theta} \left [\varphi (x, Y) \right]_{x=X}\right].
\end{equation}
We further provide concrete examples to illustrate the validity of inequality (\ref{eq:rela_indepen}). These examples demonstrate that, for specific choices of $\varphi$, the calculation results under equation (\ref{eq:subindepen}) in Definition \ref{de:subindepen} coincide with equation (\ref{eq:cla_indepen}) in Peng's framework \cite{Peng2019}.
\end{remark}

\begin{example}
\label{exam:indepen}
Given $\Omega =\{\omega_1,\omega_2\}$ and a convex domain ${\mathcal{D}}= \{\theta_1: \frac{1}{3} \le \theta_1 \le \frac{2}{3} \}$.
Let $X$ and $Y$ be two-point random variables defined on the countable state space $(\Omega, \mathcal{H})$,
$$
X(\omega)=\left\{\begin{matrix}
1, & \omega=\omega_1\\
0, & \omega=\omega_2
\end{matrix}\right.,
\quad
Y(\omega)=\left\{\begin{matrix}
0, & \omega=\omega_1\\
1, & \omega=\omega_2
\end{matrix}\right..
$$

First, we consider the case where $\varphi(x,y)= (x-\frac{1}{2})y^2$. Let $Y$ be independent of $X$ under the condition specified in equation (\ref{eq:subindepen}), then
\begin{equation}
\label{eq:<1}
\sup_{\theta\in\mathcal{D}} E_{\theta} \left[E_{\theta} \left [\varphi (x, Y) \right]_{x=X}\right]= \sup_{\theta\in\mathcal{D}} E_{\theta} \left[(X-\frac{1}{2})(1-\theta_1) \right]= \sup_{\theta_1\in \left[\frac{1}{3},\frac{2}{3}\right]} (\theta_1-\frac{1}{2}) (1-\theta_1)=\frac{1}{18}.
\end{equation}
Let $Y$ be independent of $X$ under the condition specified in equation (\ref{eq:cla_indepen}), then
\begin{equation}
\label{eq:<2}
\sup_{\theta\in\mathcal{D}} E_{\theta} \left[\sup_{\theta\in\mathcal{D}} E_{\theta} \left [\varphi (x, Y) \right]_{x=X}\right]=\sup_{\theta_1 \in \left[\frac{1}{3},\frac{2}{3}\right]} \left[ (\frac{1}{2} \sup_{\theta_1 \in \left[\frac{1}{3}, \frac{2}{3}\right]} (1-\theta_1))\theta_1 + (\frac{1}{2} \sup_{\theta_1 \in \left[\frac{1}{3},\frac{2}{3}\right]}(\theta_1-1)) (1-\theta_1)\right]=\frac{1}{6}.
\end{equation}
Comparing equations (\ref{eq:<1}) and (\ref{eq:<2}), for $\varphi(x,y)= (x-\frac{1}{2})y^2$, we have
$$
\sup_{\theta\in\mathcal{D}} E_{\theta} \left[E_{\theta} \left [\varphi (x, Y) \right]_{x=X}\right] < \sup_{\theta\in\mathcal{D}} E_{\theta} \left[\sup_{\theta\in\mathcal{D}} E_{\theta} \left [\varphi (x, Y) \right]_{x=X}\right].
$$

Second, we consider the case where $\varphi(x,y)=x(1-y)$. Let $Y$ be independent of $X$ under the condition specified in equation (\ref{eq:subindepen}), then
\begin{equation}
\label{eq:=1}
\sup_{\theta\in\mathcal{D}} E_{\theta} \left[E_{\theta} \left [\varphi (x, Y) \right]_{x=X}\right]= \sup_{\theta_1\in \left[\frac{1}{3}, \frac{2}{3}\right]} E_{\theta} \left[X\theta_1\right]= \sup_{\theta_1\in \left[\frac{1}{3}, \frac{2}{3}\right]} \theta^2_1 =\frac{4}{9}.
\end{equation}
Let $Y$ be independent of $X$ under the condition specified in equation (\ref{eq:cla_indepen}), then
\begin{equation}
\label{eq:=2}
\sup_{\theta\in\mathcal{D}} E_{\theta} \left[\sup_{\theta\in\mathcal{D}} E_{\theta} \left [\varphi (x, Y) \right]_{x=X}\right]
= \sup_{\theta\in \mathcal{D}} E_{\theta} \left[X (\sup_{\theta_1\in \left[\frac{1}{3}, \frac{2}{3}\right]} \theta_1)\right] = \frac{2}{3} \sup_{\theta_1\in \left[\frac{1}{3}, \frac{2}{3}\right]}E_{\theta}[X] =\frac{4}{9}.
\end{equation}
Comparing equations (\ref{eq:=1}) and (\ref{eq:=2}), for $\varphi(x,y)=x(1-y)$, we have
$$
\sup_{\theta\in\mathcal{D}} E_{\theta} \left[E_{\theta} \left [\varphi (x, Y) \right]_{x=X}\right] = \sup_{\theta\in\mathcal{D}} E_{\theta} \left[\sup_{\theta\in\mathcal{D}} E_{\theta} \left [\varphi (x, Y) \right]_{x=X}\right].
$$
\end{example}

Based on the independence under each probability $P_{\theta}$, we can prove the following law of large numbers by Dominated convergence theorem.

\begin{theorem}[Law of large numbers]
\label{theo:largenum}
Let $\{X_m\}_{m\ge 1}$ be a sequence of random variables defined on the countable state space $(\Omega,\mathcal{H})$. We further assume $\{X_m\}_{m=1}^{\infty}$ is an independent sequence under each probability $P_{\theta}$ with the same upper and lower expectations, i.e. $-\mathbb{E}[-X_m]=\underline{\mu},\ \mathbb{E}[X_m]=\overline{\mu}$.
Then
\begin{equation}
\label{eq:largenum}
\lim_{n\to\infty} \mathbb{E} \left [ \varphi \left ( \frac{X_1+\cdots+X_n}{n} \right ) \right ] = \sup_{\mu \in [\underline{\mu}, \overline{\mu}]} \varphi(\mu),\quad \varphi\in C_{b.lip}(\mathbb{R}).
\end{equation}
\end{theorem}

\begin{proof}
Since $\varphi\in C_{b.lip}(\mathbb{R})$, there exist constant $M>0$ such that  $|\varphi|\le M < \infty$.
Owing compact condition is stronger than relatively compact, then Theorem $\ref{theo:control}$ can also be used in the compact region $\mathcal{D}$.
By applying Theorem \ref{theo:control} and classical Dominated convergence theorem, we have
\begin{equation}
\label{eq:law1}
\lim_{n\to\infty} \mathbb{E} \left [ \varphi \left ( \frac{X_1+\cdots+X_n}{n} \right ) \right ]
=\mathbb{E} \left [\lim_{n\to\infty} \varphi \left(\frac{X_1+\cdots+X_n}{n} \right ) \right ]
=\sup_{\theta\in\mathcal{D}} \lim_{n\to\infty} E_{\theta} \left [ \varphi \left(\frac{X_1+\cdots+X_n}{n} \right ) \right ].
\end{equation}
For each $P_{\theta}$, by the classical law of large numbers, we have that $\{\frac{1}{n}(X_1+\cdots+X_n)\} \overset{P_{\theta}}{\to} \mu_{\theta}$ for all $\theta\in\mathcal{D}$. Founded in the classical Dominated convergence theorem, it is apparent that
$$
\lim_{n\to\infty} E_{\theta} \left [ \varphi \left(\frac{X_1+\cdots+X_n}{n} \right ) \right ]= \varphi (\mu_{\theta}).
$$
Thus
\begin{equation}
\label{eq:law2}
\sup_{\theta\in\mathcal{D}}\lim_{n\to\infty} E_{\theta} \left [ \varphi \left(\frac{X_1+\cdots+X_n}{n} \right ) \right ]= \sup_{\theta\in \mathcal{D}} \varphi (\mu_{\theta}),
\end{equation}
where $\mu_{\theta}= E_{\theta} [X_1]$.
Note that $\mathcal{D}$ is convex and compact, infimum and supremum of $E_{\theta}[X_i]$ can be attained at some $\theta\in\mathcal{D}$, i.e. $\mu_{\underline{\theta}}= \inf_{\theta\in\mathcal{D}} E_{\theta}[X_1]= \underline{\mu}$ and  $\mu_{\overline{\theta}}= \sup_{\theta\in\mathcal{D}} E_{\theta}[X_1]= \overline{\mu}$.
And for each $\mu\in [\underline{\mu}, \overline{\mu}]$, there exists $\theta\in \mathcal{D}$ such that $\mu_{\theta}=\mu$. Then, we have
\begin{equation}
\label{eq:law3}
\sup_{\theta\in \mathcal{D}} \varphi (\mu_{\theta})
=\sup_{\mu\in [\underline{\mu}, \overline{\mu}]} \varphi(\mu).
\end{equation}
Combing equations (\ref{eq:law1}) (\ref{eq:law2}) and (\ref{eq:law3}), it follows that
$$
\lim_{n\to\infty} \mathbb{E} \left [ \varphi \left ( \frac{X_1+\cdots+X_n}{n} \right ) \right ]
=\sup_{\theta\in\mathcal{D}} \lim_{n\to\infty} E_{\theta} \left [ \varphi \left(\frac{X_1+\cdots+X_n}{n} \right ) \right ]
= \sup_{\mu\in [\underline{\mu}, \overline{\mu}]} \varphi(\mu).
$$
This completes the proof.
\end{proof}

\begin{remark}
Note that in Peng \cite{Peng2019}, a maximal distribution is defined as
$$
\mathbb{E}[\varphi(\xi)]= \sup_{\mu\in[\underline{\mu}, \overline{\mu}]} \varphi (\mu),\quad \forall \varphi\in C_{b.lip}(\mathbb{R}).
$$
Theorem \ref{theo:largenum} shows that a sequence $\left \{\frac{1}{n}(X_1+\cdots+X_n)\right \}$ converges to a maximal distribution in law under countable state space, which is consistent with the nonlinear law of large numbers in Peng \cite{Peng2019}.
\end{remark}

\begin{remark}
According to Peng and Jin \cite{Peng2021}, the maximum estimator is the largest unbiased estimator for the upper mean and the minimum  estimator is the smallest unbiased estimator for the lower mean. Based on Theorem \ref{theo:largenum}, we can use the moment estimation to estimate upper and lower expectations. The sample moment $\beta=\frac{1}{n}(X_1+\cdots+X_n)$  converges to the maximal distribution with parameters $\hat{\underline{\mu}}$ and $\hat{\overline{\mu}}$, thus
$$
\hat{\underline{\mu}}=\inf_{\theta\in \mathcal{D}} E_{\theta} \left[\frac{X_1+\cdots+X_n}{n} \right],\quad
\hat{\overline{\mu}}=\sup_{\theta\in \mathcal{D}} E_{\theta} \left[\frac{X_1+\cdots+X_n}{n} \right],
$$
which indicates that the mean uncertainty can be estimated by the infimum and supremum of the expectation of the sequence.
\end{remark}

\section{Conclusion}
\label{sec:conclude}
In this paper, we develop a framework for sublinear expectation under countable state space. Building upon nonlinear randomized experiments, we introduce a countable state space and characterize a family of probability measures though a convex compact domain $\mathcal{D}$. Leveraging iterative summation techniques, we derive an explicit calculation for sublinear expectation and illustrate its application through representative examples.
By relative compactness probability sets $\mathcal{P}_{\Theta}$, we present Monotone Convergence Theorem, Fatou's Lemma, and Dominated Convergence Theorem under countable state space.
Furthermore, we establish a law of large numbers of sublinear expectation by combining the independence under each $P_{\theta}$ with the Dominated Convergence Theorem. This result demonstrates that the random sequence converges to a maximal distribution.
Potential applications of this framework to uncertainty modeling, financial risk quantification and derivatives pricing are discussed as directions for future researches.

\bibliography{gexpf}

\end{document}